\documentclass[11pt,a4paper]{scrartcl}

\usepackage{mystyle-koma-enhanced}
\usepackage{tikz}
\usetikzlibrary{patterns}
\usetikzlibrary{decorations.pathmorphing}
\definecolor{albiceleste}{RGB}{130,130,245}
\definecolor{skyblue}{RGB}{210,220,255}
\tikzset{vertex/.style={circle,draw,inner sep=0pt,minimum size=6pt}}

\renewcommand{\P}{\mathbb{P}}
\newcommand{\E}{\mathbb{E}}
\newcommand{\N}{\mathbb{N}}
\newcommand{\Z}{\mathbb{Z}}
\newcommand{\R}{\mathbb{R}}
\newcommand{\1}{\mathbb{1}}

\newcommand{\dG}{\operatorname{d}_{G}}
\newcommand{\Reff}{\mathcal{R}_{\mathrm{eff}}}

\crefname{assumption}{assumption}{assumptions}
\Crefname{assumption}{Assumption}{Assumptions}


\title{Recurrence of strong-decay inhomogeneous long-range percolation clusters}
\author{Johannes B\"aumler \orcidlink{0000-0002-3823-9899} \and Lukas L\"uchtrath \orcidlink{0000-0003-4969-806X} \and Christian M\"onch \orcidlink{0000-0002-6531-6482}}
\date{August 2026}

\hypersetup{
  pdftitle={Recurrence of strong-decay inhomogeneous long-range percolation clusters},
  pdfauthor={Johannes Bäumler, Lukas Lüchtrath, Christian Mönch}
}


\begin{document}
\maketitle

\begin{abstract}
\noindent We prove recurrence criteria for inhomogeneous long-range percolation in
dimensions one and two.  In dimension one, recurrence follows from a purely
geometric scarcity condition: long edges eventually disappear on exponential
scales.  This applies to weight-dependent random connection models and related
one-dimensional spatial scale-free graphs whenever the standard strong-decay
long-edge estimate holds.  In dimension two, we combine the linear
chemical-distance estimate of L\"uchtrath with an area-order bound on the
degree measure.  Graph-distance layers in exponentially separated bands then
give the required Nash--Williams cutsets for planar random geometric graphs
satisfying the polynomial mixing and long-edge estimates
[\emph{J.\ Theoret.\ Probab.}~\textbf{39} (2026), Paper No.~12].  As a concrete consequence, every
connected component of the two-dimensional
weight-dependent random connection model with interpolation kernel is
recurrent throughout the strong-decay region
\[
  \delta>2,\qquad
  \gamma<1-\frac1\delta,\qquad
  \alpha<1-\gamma.
\]
\end{abstract}

\section{Introduction}
\label{sec:introduction}

The recurrence and transience of random walks on long-range percolation
clusters is sensitive to dimension, degree tails, and the abundance of long
edges.  In homogeneous long-range percolation, Berger \cite{Berger2002} showed that the return properties of the simple random walk on the infinite cluster depend on the asymptotic probability of seeing long edges, $\mathbb{P}(x\sim y)$, with the transition between the recurrent and the transient regime happening when $\mathbb{P}(x\sim y) \approx \|x-y\|^{-2d}$ in dimensions one and two.  The question becomes more
delicate for inhomogeneous long-range models, such as spatial scale-free
percolation and the weight-dependent random connection model (WDRCM), because the
vertex weights may create heavy-tailed degrees even when very long edges are scarce
\cite{GHMM2022,DeijfenHofstadHooghiemstra2013,KrioukovEtAl2010,
BringmannKeuschLengler2019,GGLM2019,GLM2021,GGM22}.

In this note, we give sufficient conditions for recurrence based on the absence of long edges and the tail decay of the degree distribution in dimensions one and two. The regimes in which we establish recurrence are essentially optimal, since transience was established slightly outside these regimes in \cite{Moench2024, GHMM2022}. The
one-dimensional result is an electrical-network criterion on the line.  If,
almost surely, only finitely many large dyadic annuli contain an edge crossing over a fixed cut, then there will, almost surely, only be finitely many edges crossing this cut. By translation invariance, this will occur infinitely often almost surely, and an application of the Nash-Williams inequality shows that the effective resistance between $0$ and $\infty$ equals $+\infty$ almost surely, implying the recurrence of the simple random walk.
This argument does not require independence of the edges, and applies directly to
WDRCMs satisfying the strong-decay long-edge estimates from
\cite{GracarLuechtrathMoench2025, jacobJahLu2024}, thereby completing the phase diagram away from the open boundaries left by~\cite{GHMM2022,Moench2024,Baeumler2023_recurrence}.

The argument in dimension $d=2$ is different.  The recurrence follows from the existence of
infinitely many annuli $\left\{x \in \R^2 : r \leq |x| \leq 2r\right\}$ in which the graph distance across the
annulus is linear in its Euclidean width and the number of edges in the inner
ball is of order area.  The resulting graph-distance layers form
edge-disjoint cutsets, so the Nash--Williams criterion gives a
uniform resistance contribution at each such good scale.  We verify the existence of such annuli
for any stationary planar random geometric graph satisfying the polynomial
mixing and long-edge estimates used by the second author in the long-range
chemical-distance work \cite{lue2024spatialrandomgraphsweak}.
For the so-called
interpolation kernel, the required strong-decay region is exactly the
negative-\(\zeta\) phase identified by Jacob et al.\
\cite{jacobJahLu2024}.

\section{Main Results}
\label{sec:main-results}

We regard every graph as an electrical network with unit conductances unless
conductances are explicitly specified.  A connected component is recurrent if
the effective resistance from one, equivalently every, vertex of that
component to infinity is infinite.  Finite components are recurrent by
convention.

For a set \(A\), write \(A^{[2]} \coloneqq \left\{ \{x,y\} \subset A : x \neq y  \right\}\) for the unordered pairs of distinct
elements of \(A\).  A random geometric graph \(G=(V,E)\) in \(\R^d\) is a random graph with $V\subset \R^d, E \subset V^{[2]}$. We say that this graph is locally finite if, almost surely, \(V\cap K\) is finite for every
compact \(K\subset\R^d\) and every vertex has finite graph degree.  In
particular, \(V\) is almost surely countable.  

Our motivating example is the weight-dependent random connection model (WDRCM) that is defined as follows. The vertex set \(V\) is a standard Poisson point process on \(\R^d\) of intensity \(\lambda>0\). We identify each vertex with its location \(v\in\R^d\). Each vertex has a mark \(u_v\in(0,1)\), and, conditional on $V$, the marks $(u_v)_{v \in V}$ are i.i.d.\ random variables that are uniformly distributed on $(0,1)$. Then, we choose a decreasing \emph{profile} function \(\rho\colon (0,\infty)\to[0,1]\) and a symmetric and coordinate-wise decreasing \emph{kernel} function \(g:(0,1)^2\to(0,\infty)\) and, given a pair of vertices \(v,w\) and their marks \(u_v,u_w\), we connect them with probability \(\rho(g(u_v,u_w)|v-w|^d)\), independent of all other connections. A kernel of particular interest is the \emph{interpolation kernel} \(g(u_v,u_w)=(u_v\wedge u_w)^\gamma (u_v\vee u_w)^\alpha\), where \(\alpha,\gamma\geq 0\) are parameters, whereas the profile is typically chosen to be either \(\rho(x)=1\wedge x^{-\delta}\), for some \(\delta>1\), or \(\rho=\1_{[0,1]}\). In previous literature, the same model was also considered with $V=\Z^d$ instead of $\R^d$, cf. \cite{berger2004,Berger2002,Baeumler2023_recurrence}.

In dimension \(d=2\) there always exists an infinite cluster if \(\lambda\) is large enough~\cite{GLM2021}. This is not always the case in dimension \(d=1\), where the existence of an infinite cluster was studied in~\cite{GracarLuechtrathMoench2025}. If there does not exist an infinite cluster in dimension $d=1$, we augment the graph by nearest-neighbour edges in the following sense: We assume that the vertices of the Poisson processes are ordered by location, i.e.\ \(V=(v_i)_{i\in\Z}\) with \(v_i<v_j\) if \(i<j\), and add all edges \(\{v_i,v_{i+1}\}\), \(i\in\Z\) to the graph. 
Note that when $V = \Z$, one can also always enforce the existence of an infinite cluster by considering $\rho$ and $g$ with $\rho(g(s,t)\cdot 1)=1$ for all $s,t \in (0,1)$.


\paragraph{Recurrence in dimension \(\boldsymbol{d=1}\). }
Consider a random geometric graph on the line, i.e.\ \(V\subset\R\) and define the number
of edges above the cut \(0+\) by
\[
  \mathsf X
  \coloneqq 
  \sum_{x\in V:x\le 0}
  \sum_{y\in V:y>0}
  \1\{\{x,y\}\in E\}.
\]

\begin{theorem}[Recurrence in dimension one]
\label{thm:one-dimensional-criterion}
Let \(G=(V,E)\) be a locally finite random geometric graph with
\(V\subset\R\).  Assume that
its law is invariant and ergodic under translations by \(\Z\).  If
\[
  \P(\mathsf X<\infty)>0,
\]
then every connected component of \(G\) is recurrent almost surely.
\end{theorem}

We shall use the following easy-to-check sufficient condition.  We say that
long edges eventually disappear at exponential scales if
\begin{equation}
\label{eq:exponential-long-edge-disappearance}
  \sum_{k\ge 1}
  \P\big(
    \exists x\in V\cap[-2^k,2^k],\,y\in V:
    |x-y|\ge 2^k,\ \{x,y\}\in E
  \big)<\infty.
\end{equation}

\begin{corollary}[Strong-decay graphs on the line]
\label{cor:one-dimensional-long-edges}
Let \(G=(V,E)\) be as in \Cref{thm:one-dimensional-criterion}.  If
\eqref{eq:exponential-long-edge-disappearance} holds, then every connected
component of \(G\) is recurrent almost surely.
\end{corollary}

\begin{corollary}[One-dimensional WDRCM]
\label{cor:one-dimensional-wdrcm}
Consider the one-dimensional WDRCM augmented by nearest-neighbour edges with profile function \(\rho(x)=1\wedge x^{-\delta}\) and the interpolation kernel. If 
\[
  \delta>2,\qquad
  \gamma<1-\frac1\delta,\qquad
  \alpha<1-\gamma,
\]
%
%
then the graph, equivalently its unique infinite component, is recurrent
almost surely.
\end{corollary}

\paragraph{Recurrence in dimension \(\boldsymbol{d=2}\). }
We now turn to dimension two.  Let \(G=(V,E)\) be a locally finite random
geometric graph in \(\R^2\).  We write \(\P\) for the stationary law
of the graph.  For \(z\in\R^2\) and \(r>0\), set
\[
  \Lambda_r(z)\coloneqq \{v\in V: |v-z|\le r\},
  \qquad
  \Lambda_r(z)^{\mathsf c}\coloneqq V\setminus \Lambda_r(z).
\]
For a finite set \(A\subset V\),
let \(T(A)\coloneqq \sum_{v\in A}\deg_G(v)\) be the sum of the degrees of vertices inside $A$, and for disjoint sets $A,B \subset \R^d$, let
\begin{equation*}
	d_G(A,B) = \inf\left\{n \geq 1 : \exists x_0,\ldots,x_n \in V \text{ with } x_0 \in A, x_n \in B, \{x_i,x_{i+1}\} \in E \text{ for all }i\right\}
\end{equation*}
be the graph distance between $A$ and $B$ in the graph $G$.

In the following, we use the following parts of the so-called \emph{PM/PL framework} of
\cite{lue2024spatialrandomgraphsweak}.  Here,
\(\mathrm{PM}\) stands for polynomial mixing and \(\mathrm{PL}\) for
polynomial long-edge decay.  For \(m>0\), write
\(Q_m(z)=z+[-m/2,m/2)^2\).  A box event is local if it is determined by the
vertices, marks, and internal edges of the box.  The property
\(\mathrm{PM}_\xi\) means that, for every local box event \(E\), there is a
constant \(C_E<\infty\) such that, for all large \(m\),
\[
  \sup_{|z|>2}
  \operatorname{Cov}\big(
    \1_{E(Q_m(0))},\1_{E(Q_m(mz))}
  \big)
  \le C_E m^\xi.
\]
Thus correlations between two boxes of side length \(m\), separated by a
distance of order \(m\), are \(O(m^\xi)\).  The property
\(\mathrm{PL}_\mu\) means that for the event
\[
  L(m,n)\coloneqq \{\exists x\sim y:x,y\in Q_m(0),\ |x-y|>n\},
\]
there is a constant \(C_L<\infty\) such that, for all \(m\ge1\) and all large enough
\(n \geq 1\),
\[
  \P(L(m,n))\le C_L m^2 n^\mu.
\]
In words, the probability that a box of area order \(m^2\) contains an
internal edge longer than \(n\) is bounded by the corresponding polynomial
tail.  We only use these properties with \(\xi<0\) and \(\mu<-2\).

\begin{assumption}[Planar PM/PL condition]
\label{ass:planar-pmpl-input}
The stationary, ergodic, locally finite random geometric graph \(G\) in
\(\R^2\) satisfies the following conditions.
\begin{enumerate}[label=(\roman*)]
\item The law satisfies \(\mathrm{PM}_\xi\) and \(\mathrm{PL}_\mu\) for some
\(\xi<0\) and \(\mu<-2\).
\item The degree-measure intensity is finite, i.e.\
\[
  \theta_M\coloneqq \E\big[T\left(V\cap [0,1)^2 \right)\big]<\infty.
\]
\end{enumerate}
\end{assumption}

\begin{theorem}[Recurrence in dimension two]
\label{thm:planar-rec-d2}
Let \(G\) be a random geometric graph with stationary and ergodic law $\P$ satisfying \Cref{ass:planar-pmpl-input}. Then almost surely every connected component of \(G\) is recurrent.
\end{theorem}

This theorem has a direct implication for the WDRCM in dimension $d=2$, which was also the motivation behind \Cref{thm:planar-rec-d2}. Away from the open phase boundary, the implication completes the phase diagram for simple random walk on the infinite cluster of the WDRCM, previously studied in \cite{GHMM2022,Baeumler2023_recurrence,Moench2024}. The result of the following corollary is also illustrated in \Cref{figure1}.

\begin{corollary}[Interpolation kernel WDRCM in dimension two]
\label{cor:wdrcm-interpolation-d2}
Consider the WDRCM with interpolation kernel $\left( g(s,t)=(s\wedge t)^\gamma (s\vee t)^\alpha \right)$ in
dimension \(d=2\), with 
\(\rho(x)\asymp 1\wedge x^{-\delta}\).  If
\[
  \delta>2,\qquad
  \gamma<1-\frac1\delta,\qquad
  \alpha<1-\gamma,
\]
then 
almost surely every connected
component is recurrent.
\end{corollary}





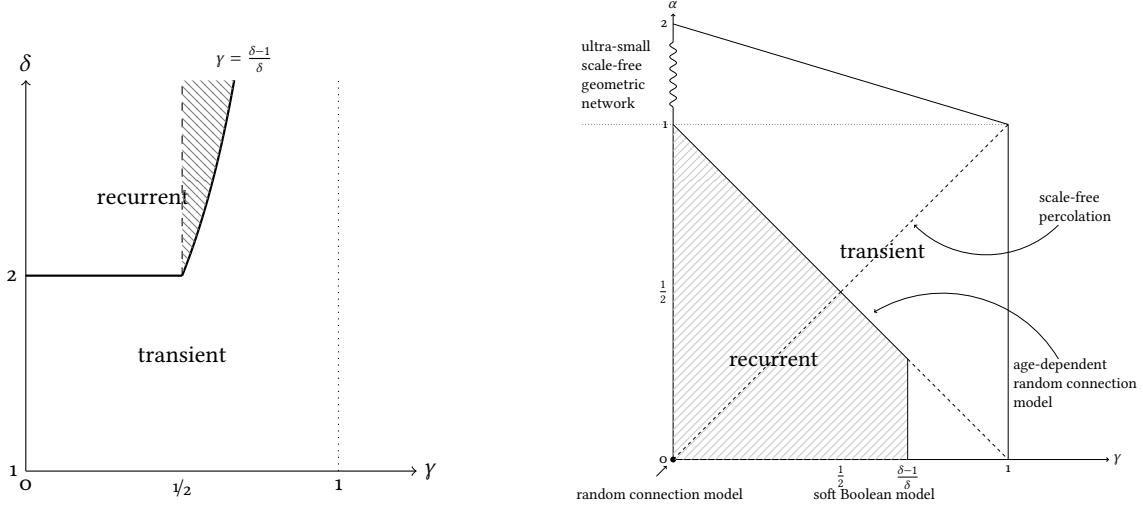
\begin{figure}[ht!]
\centering
\begin{subfigure}{0.4\textwidth}
    \resizebox{\textwidth}{!}{
\begin{tikzpicture}[every node/.style={scale=0.8}]
\draw[->] (0,0) -- (5,0) node[right] {$\gamma$};
\draw	(0,0) node[anchor=north] {0}
(2,0) node[anchor=north] {\nicefrac{1}{2}}
(4,0) node[anchor=north] {1};
 \draw (2.8,5) node[above] {{\scriptsize $\gamma=\frac{\delta-1}{\delta}$}};

\draw[->] (0,0) -- (0,5) node[above] {$\delta$};
\draw (0,0) node[anchor=east] {1}
(0,2.5) node[anchor=east] {2};

\draw[dotted] (4,0) -- (4,5);

\draw[dashed] (2,2.5) -- (2,5);
\draw[thick] (0,2.5) -- (2.0,2.5);
\draw (2,1.5) node {transient};
\draw (1.5,3.5) node {recurrent};
\draw[thick] (0,0) plot[domain=0.5:2/3,variable=\g]({4*\g},{-2.5+2.5*1/(1-\g)});
\draw[pattern=north west lines, pattern color=gray, draw=none] 
   (0,0) plot[smooth,samples=200,domain=2/3:0.6665,variable=\g]({2},{-2.5+2.5*1/(1-\g)}) -- 
    plot[smooth,samples=200,domain=2/3:0.5,variable=\g]({4*\g},{-2.5+2.5*1/(1-\g)});
\end{tikzpicture}
}
\end{subfigure}
\hfill
\begin{subfigure}{0.5\textwidth}
    \resizebox{\textwidth}{!}{
    \begin{tikzpicture}[every node/.style={scale=1.2}]
        \draw[->] (0,0) to (13,0) node[right] {$\gamma$};
        \draw	(5,0) node[anchor=north] {$\tfrac{1}{2}$}
    		  (10,0) node[anchor=north] {1};

        \draw[dotted] (-2.7,10) to (10,10);
        \draw[] (10,0) to (10,10)
	           (10,10) to (0,13);

        \draw[](0,0) to (0,10.5);
        \draw [->] (0,12.5) to (0,13.3) node[above] {$\alpha$};
        \draw[decorate, decoration = {snake, segment length = 10 pt, amplitude = 1mm}] (0, 10.5)--(0,12.5);
        \draw	(0,0) node[anchor=east] {0}
        	   (0,5) node[anchor=east] {$\tfrac{1}{2}$}
        	   (0,10) node[anchor=east] {1}
        	   (0,13) node[anchor = east] {2};

        \draw (6,-0.7) node[align = left, anchor = north] {soft Boolean model};
        \draw (-1.7,11.5) node[align = left] {ultra-small \\ scale-free \\ geometric \\ network};
        \draw (12,2.3) node[align = left] {age-dependent \\ random connection \\ model};
        \draw[->, bend angle = 45, bend right] (11.5,3) to (6,4.4);
        \draw (12,7.5) node[align = left] {scale-free \\ percolation};
        \draw[->, bend angle = 45, bend left] (11.5,7) to (7.2,7);
        \draw (0,0) node[circle, fill = black, scale=0.4] {};
        \draw (-0.4,-0.7) node[align = left, anchor = north] {random connection model};
        \draw[->] (-0.5, -0.5) to (-0.2,-0.2);

         \draw 	(7, 0) node[anchor = north] {$\tfrac{\delta-1}{\delta}$};

        \draw[pattern=north east lines, pattern color=lightgray!80!, draw=none] 
            (0,0) plot[smooth,samples=200,domain=0:7,variable=\g]({0},{10*\g/7}) -- 
            plot[smooth,samples=200,domain=7:0,variable=\g]({7},{3*\g/7});
         
   	    \draw[thick] (7,0) to (7,3); 
	   \draw[thick] (0,10) to (7,3); 
        \draw[dashed] (7,3) to (10,0); 
        \draw[dashed] (0,0) to (3,3); 
        \draw[dashed] (3,3) to (5,5); 
        \draw[dashed] (5,5) to (10,10); 
 

              \draw	(6.25,6.25) node[scale = 1.5, thick] {transient};
              \draw	(3,3) node[scale = 1.5, thick] {recurrent};
    \end{tikzpicture}
    }
\end{subfigure}
\caption{Phase diagrams for the WDRCM in dimension \(d=2\) with the min kernel ($g(s,t)=(s\wedge t)^\gamma$; soft Boolean model) on the left or the interpolation kernel with some models it represents on the right. The behaviour on the phase boundary remains open. The gray area on the left is newly established through our results. The other phases are already established in~\cite{GHMM2022,Baeumler2023_recurrence,Moench2024}. On the right, only \(\delta>2\) is shown. The gray recurrent phase is due to our work, the transience phase is due to~\cite{Moench2024}.}
\label{figure1}
\end{figure}

\section{Proofs}
\label{sec:proofs}

\subsection{Dimension one}
\label{sec:proofs-d1}

\begin{proof}[Proof of \Cref{thm:one-dimensional-criterion}]
By continuity from below and the assumption,
\begin{align*}
	\lim_{m \to \infty} \P(\mathsf X \leq m , |a|< m,|b| < m \text{ for all } \{a,b\} \in E \text{ with } a \leq 0, b > 0) = \P(\mathsf X< \infty) > 0.
\end{align*}
In particular, there exists $m \in \N$ for which the probability on the left-hand side is positive. For $j \in \Z$, define
\begin{multline*}
	A_j \coloneqq \Big\{ \sum_{x \in V : x \leq j} \sum_{y \in V : y > j} \mathbf{1}\left\{\{x,y\} \in E \right\} \leq m ,
	\\
	|a-j|< m, |b-j| < m \text{ for all } \{a,b\} \in E \text{ with } a \leq j, b > j \Big\}.
\end{multline*}
The events $(A_j)_{j\in\Z}$ are translates of $A_0$. Hence the ergodic theorem for the unit translation and its inverse gives
\begin{align}\label{eq:clean-cut-frequency}
	\lim_{\ell \to \infty} \frac{1}{\ell} \sum_{j=1}^{\ell} \mathbf{1}\left\{A_j\right\}
	= \lim_{\ell \to \infty} \frac{1}{\ell} \sum_{j=-\ell}^{-1} \mathbf{1}\left\{A_j\right\}
	= \varepsilon_m \quad \text{almost surely},
\end{align}
where
\[
\varepsilon_m = \P(\mathsf X \leq m , |a|< m,|b| < m \text{ for all } \{a,b\} \in E \text{ with } a \leq 0, b > 0) > 0.
\]
Let $v\in V$. By \eqref{eq:clean-cut-frequency}, we may choose strictly increasing sequences $(j_i)_{i\in \N}$ and $(\tilde{j}_i)_{i\in \N}$ in $\N$ such that
\[
|v|+m<j_1\wedge\tilde{j}_1,
\qquad j_{i+1}-j_i>2m,
\qquad \tilde{j}_{i+1}-\tilde{j}_i>2m,
\]
and all the events $A_{j_i}$ and $A_{-\tilde{j}_i}$ occur. Define
\begin{align*}
	F_i \coloneqq \Big\{ \{x,y\} \in E : x \leq j_i, y > j_i \Big\}
	\cup \Big\{ \{x,y\} \in E : x \leq -\tilde{j}_i, y > -\tilde{j}_i\Big\}, \quad i \in \N.
\end{align*}
Every edge in $F_i$ has both endpoints in one of the intervals $(j_i-m,j_i+m)$ and $(-\tilde{j}_i-m,-\tilde{j}_i+m)$. Our choice of the two sequences therefore makes the sets $F_i$ pairwise edge-disjoint. Moreover, each $F_i$ separates $v$ from infinity and satisfies $|F_i|\leq 2m$. Indeed, local finiteness forces every infinite path starting at $v$ to leave the bounded interval $[-\tilde{j}_i,j_i]$, and its first exit crosses one of the two cuts defining $F_i$.

If some $F_i$ is empty, then the component of $v$ is finite and hence recurrent. Otherwise, the Nash--Williams inequality implies that
\begin{align*}
	\Reff(v,\infty) \geq \sum_{i=1}^{\infty} \frac{1}{|F_i|}
	\geq \sum_{i=1}^{\infty} \frac{1}{2m} = \infty,
\end{align*}
showing that the vertex $v$ is recurrent. Since $v\in V$ was arbitrary, this finishes the proof.
\end{proof}

\begin{proof}[Proof of \Cref{cor:one-dimensional-long-edges}]
By \Cref{thm:one-dimensional-criterion}, it suffices to prove
\(\mathsf X<\infty\) with positive probability, and in fact the argument gives that $\mathsf X<\infty$ almost surely.  Let \(A_k\) be the event in the \(k\)-th summand of
\eqref{eq:exponential-long-edge-disappearance}.  By Borel--Cantelli, only
finitely many \(A_k\) occur almost surely.

Let \(\{x,y\}\) be an edge crossing \(0+\), with \(x\le0<y\), and put
\(\ell=|x-y|\).  If \(\ell\ge2\), choose \(k\) such that
\(2^k\le \ell<2^{k+1}\).  Since \(|x|+y=\ell\), at least one endpoint lies in
\([-2^k,2^k]\), while the edge has length at least \(2^k\).  Hence \(A_k\)
occurs.  As only finitely many \(A_k\) occur almost surely, all crossing edges
have uniformly bounded length, say that $|x-y| \leq 2^K$ for all edges $\{x,y\}\in E$ with $x\leq 0 < y$.  By local finiteness, the total number of edges $\{x,y\}$ with $|x|, |y| \leq 2^{K}$ is finite almost surely, showing that $\P(\mathsf X<\infty)=1$.
\end{proof}

\begin{proof}[Proof of \Cref{cor:one-dimensional-wdrcm}]
The chosen parameter regime is precisely the strong decay regime corresponding to \(\zeta<0\) in~\cite{jacobJahLu2024}. By definition therein, this implies the existence of constants \(c,\vartheta>0\) such that, for all \(n\ge1\),
\[
  \P\big(
    \exists x\in V\cap[-n,n],\,y\in V:
    |x-y|\ge n,\ \{x,y\}\in E
  \big)\le c n^{-\vartheta}.
\]
Evaluating
the estimate at \(n=2^k\) gives the summability condition
\eqref{eq:exponential-long-edge-disappearance}.  The claim follows from
\Cref{cor:one-dimensional-long-edges}.  The deterministic nearest-neighbour
edges make the graph connected, so the recurrent component is the whole graph.
\end{proof}

\subsection{Dimension two}
\label{sec:proofs-d2}

\begin{proposition}
\label{prop:fixed-box-distance-rec}
Suppose that \(G\) is stationary, ergodic, and locally finite, that its degree
measure has finite intensity, and that for all $L \geq 0$ there exists \(\eta>0\) such that
\begin{equation}
\label{eq:fixed-box-distance}
  \P\big(\mathcal D(R)\big)\longrightarrow1
  \qquad\text{as }R\to\infty,
\end{equation}
where
\[
  \mathcal D(R)\coloneqq 
  \left\{
    \dG(v,w)\ge\eta|v-w|
    \text{ for all }v\in V\cap Q_L(0),\
    w\in V\cap Q_R(0)^{\mathsf c}
  \right\}.
\]
Then almost surely every connected component of
\(G\) is recurrent.
\end{proposition}

\begin{proof}
If \(\theta_M=0\), then the non-negative random variable
\(T(V\cap[0,1)^2)\) vanishes almost surely. By stationarity, the same is true
in every integer translate of the unit box, so every vertex has degree zero
almost surely and the claim is immediate. We may therefore assume that
\(\theta_M>0\).

Fix \(L\in\N\), and let \(\eta>0\) be as in
\eqref{eq:fixed-box-distance}. We prove the claim for all components meeting
\(Q_L(0)\); a countable intersection over \(L\) will then give the result.
Set \(R_n=4^n\).  Since
\(\P(\mathcal D(R_n))\to1\), we get that
\[
  \P\big(\mathcal D(R_n)\text{ infinitely often}\big) = \P\left( \bigcap_{m=1}^{\infty} \bigcup_{n= m}^{\infty} \mathcal D(R_n)\right)
  \geq 
  \lim_{m\to \infty} \sup_{n \geq m}
  \P\left( \mathcal D(R_n)\right)=1.
\]

Let \(M(A)=T(V\cap A)\) be the degree measure.  The multiparameter ergodic
theorem gives
\[
  \frac{M(Q_R(0))}{R^2}\longrightarrow\theta_M
  \qquad\text{almost surely, for } R\to \infty.
\]
Thus we see that there exist, almost surely, infinitely many $n \in \N$ so that
the event \(\mathcal D(R_n)\) occurs and
\begin{equation}
\label{eq:direct-degree-volume}
  M(Q_{R_n}(0))\le 2 \theta_M R_n^2.
\end{equation}

Work on this event and fix \(v\in V\cap Q_L(0)\). If \(\mathcal D(R_n)\)
occurs and \(R_n\ge2L\), then every
\(w\notin Q_{R_n}(0)\) satisfies
\[
  |v-w|
  \ge \frac{R_n-L}{2}
  \ge \frac{R_n}{4},
\]
and therefore
\begin{equation}
\label{eq:direct-distance-outside}
  \dG(v,w)\ge\frac{\eta R_n}{4}.
\end{equation}

Put \(a=\eta/16\).  For \(k\ge0\), define
\[
  D_k(v) \coloneqq \{w\in V:\dG(v,w)\le k\},
\]
and, for \(k\ge1\), let
\[
  F_k(v) \coloneqq
  \big\{\{x,y\}\in E:|\{x,y\}\cap D_{k-1}(v)|=1\big\}.
\]
Local finiteness implies that every \(D_k(v)\), and hence every \(F_k(v)\),
is finite.  If the component of \(v\) is infinite, each \(F_k(v)\) is a cutset
separating \(v\) from infinity.  The cutsets are pairwise edge-disjoint: the
graph distances from \(v\) of the endpoints of an edge differ by at most one.

For a scale at which \(\mathcal D(R_n)\) occurs, consider
\[
  I_n\coloneqq 
  \big\{k\in\N:
    \lceil aR_n\rceil\le k\le\lfloor2aR_n\rfloor
  \big\}.
\]
By \eqref{eq:direct-distance-outside},
\(D_k(v)\subseteq Q_{R_n}(0)\) for every \(k\in I_n\).  Thus, all edges in
\(\bigcup_{k\in I_n}F_k(v)\) have both endpoints in \(Q_{R_n}(0)\).  Since
these edge sets are disjoint, \eqref{eq:direct-degree-volume} yields
\[
  \sum_{k\in I_n}|F_k(v)|
  \le M(Q_{R_n}(0))
  \le 2 \theta_M R_n^2.
\]
Moreover, \(|I_n|\ge aR_n/2\) for all large enough \(n\). Let $N$ be large enough so that $ M(Q_{R_n}(0))\le 2 \theta_M R_n^2$ and \(|I_n|\ge aR_n/2\) for all $n\geq N$.  If some \(F_k(v)\) is
empty, then \(D_{k-1}(v)\) is a finite set with no edge leaving it, so the
component of \(v\) is finite and hence recurrent.  Otherwise,
Cauchy--Schwarz gives
\begin{equation}
\label{eq:direct-resistance-band}
  \sum_{k\in I_n}\frac1{|F_k(v)|}
  \ge
  \frac{|I_n|^2}{\sum_{k\in I_n}|F_k(v)|}
  \ge
  \frac{a^2}{8 \theta_M}
\end{equation}
for $n\geq N$.
Because \(R_{n+1}=4R_n\), the bands \(I_n\), \(n \in \N\), are pairwise disjoint for all
large \(n\).  Infinitely many of the corresponding distance events occur, so
collecting their cutsets and applying Nash--Williams gives
\[
  \Reff(v,\infty)
  \ge
  \sum_{\substack{n \geq N: \\ \mathcal D(R_n)\text{ occurs}}}
  \ \sum_{k\in I_n}\frac1{|F_k(v)|}
  \ge
  \sum_{\substack{n \geq N: \\ \mathcal D(R_n)\text{ occurs}}}
  \frac{a^2}{8 \theta_M}
  =
  \infty.
\]
Thus the component of \(v\) is recurrent. Since $v$ was arbitrary in
\(V\cap Q_L(0)\), every component meeting this box is recurrent on the same
probability-one event. Intersecting these events over \(L\in\N\) proves the
claim for all components.
\end{proof}

\begin{proof}[Proof of \Cref{thm:planar-rec-d2}]
By \cite[Theorem~1]{lue2024spatialrandomgraphsweak}, for every fixed
\(L\in\N\) there exists \(\eta>0\) such that
\[
  \limsup_{R\to\infty}
  \frac{\log\P(\mathcal D_0(R)^{\mathsf c})}{\log R}
  \le \xi\vee(2+\mu)<0.
\]
In particular, \eqref{eq:fixed-box-distance} holds.  The finite
degree-measure intensity is condition~(ii) of
\Cref{ass:planar-pmpl-input}, so
\Cref{prop:fixed-box-distance-rec} applies.
\end{proof}

\begin{proof}[Proof of \Cref{cor:wdrcm-interpolation-d2}]
By \cite[Proposition~2.3]{jacobJahLu2024}, the displayed parameter range is
precisely the \(\zeta<0\) phase of the interpolation model.  In this phase
\cite[Lemma~2]{lue2024spatialrandomgraphsweak} gives
\(\mathrm{PL}_{2(\zeta-1)}\), and \(2(\zeta-1)<-2\).  Local configurations in
disjoint boxes are independent in the Poisson model and in the original
lattice coordinates, since the vertex retention variables, marks, and edge
variables are sampled independently.  Hence the model satisfies
\(\mathrm{PM}_{-\infty}\).

It remains to verify finite degree-measure intensity.
We only consider the Poisson case; in the lattice case, the corresponding lattice sums satisfy the same bounds
up to a constant. Since
\(\rho(t)\asymp 1\wedge t^{-\delta}\) and \(\delta>2\), the profile is
integrable on \([0,\infty)\).  For \(b>0\),
\[
  \int_{\R^2}\rho(b|z|^2)\,\dd z
  =
  \frac{\pi}{b}\int_0^\infty\rho(r)\,\dd r,
\]
implying that
\begin{equation*}
    \int_0^1 \int_0^1 \int_{\R^2} \rho \left( g(s,t) |z|^2 \right) \, \dd z \, \dd s \, \dd t = \int_0^1 \int_0^1 \frac{\pi}{g(s,t)} \int_0^\infty \rho(r) \, \dd r \, \dd s \, \dd t 
    = 
    C \int_0^1 \int_0^1 \frac{\dd s \, \dd t}{(s\wedge t)^\gamma (s\vee t)^\alpha} 
\end{equation*}
with $C=\pi \int_0^\infty \rho(r) \, \dd r < \infty$.
Consequently, the expected degree of a typical vertex is
bounded by a constant times
\[
  \int_0^1\int_0^1
  \frac{\dd s\,\dd t}
       {(s\wedge t)^\gamma(s\vee t)^\alpha}
  =
  2\int_0^1\int_s^1 s^{-\gamma}t^{-\alpha}\,\dd t\,\dd s
  <\infty;
\]
the last expression is finite because \(\gamma<1\) and \(\alpha+\gamma<1\).  Campbell's formula in the
Poisson case, and the corresponding per-site identity in the lattice case,
now give \(\theta_M<\infty\).  They also show that every vertex has finite
degree almost surely, so the graph is locally finite.

Thus \Cref{ass:planar-pmpl-input} holds in the Poisson case, and
\Cref{thm:planar-rec-d2} applies.  For the unshifted lattice,
\cite[Theorem~1]{lue2024spatialrandomgraphsweak} supplies the same fixed-box
distance estimate in the lattice setting, and the proof of
\Cref{prop:fixed-box-distance-rec} applies verbatim with \(\Z^2\)-stationarity
and the \(\Z^2\) ergodic theorem.  This proves the lattice case under its
original law as well.
\end{proof}

\paragraph{Acknowledgements.}
We thank Markus Heydenreich for helpful discussions at an early stage of this work.
Lukas L\"uchtrath gratefully acknowledges financial support from the Leibniz Association within the Leibniz Junior Research Group on \emph{Probabilistic Methods for Dynamic Communication Networks} as part of the Leibniz Competition and from the Deutsche Forschungsgemeinschaft (DFG, German Research Foundation) under Germany's Excellence Strategy---The Berlin Mathematics Research Center MATH+ (EXC-2046/2, project ID: 390685689) through the project \emph{EF-MA-Sys-2} on \emph{Information Flow \& Emergent Behavior in Complex Networks}.

\printbibliography

\end{document}